\documentclass{amsart}
\usepackage{amssymb,latexsym,amsmath,amscd,graphicx,setspace,amsthm,verbatim,comment}
\usepackage[margin = 3 cm]{geometry} \input xy \xyoption{all}

\theoremstyle{plain}
\newtheorem{theorem}{Theorem}[section]
\newtheorem{proposition}[theorem]{Proposition}

\newtheorem{corollary}[theorem]{Corollary}
\newtheorem{lemma}[theorem]{Lemma}

\makeatletter
\def\th@remark{%
  \thm@headfont{\bfseries}%
  \normalfont 
  \thm@preskip\topsep \divide\thm@preskip\tw@
  \thm@postskip\thm@preskip
}
\makeatother

\theoremstyle{remark} 
\newtheorem{remark}[theorem]{Remark}
      
\theoremstyle{definition}
\newtheorem{definition}[theorem]{Definition}

\begin{document} 
\title[The special value $u=1$ of Artin-Ihara $L$-functions]
{The special value $u=1$ of Artin-Ihara $L$-functions}

\author{Kyle Hammer, Thomas W. Mattman, Jonathan W. Sands, Daniel Valli\`{e}res}

\address{Mathematics and Statistics Department, California State University, Chico, CA 95929 USA}
\email{khammer4@mail.csuchico.edu}
\address{Mathematics and Statistics Department, California State University, Chico, CA 95929 USA}
\email{tmattman@csuchico.edu}
\address{Department of Mathematics, University of Vermont, Burlington, VT 05401, USA} 
\email{Jonathan.Sands@uvm.edu} 
\address{Mathematics and Statistics Department, California State University, Chico, CA 95929 USA}
\email{dvallieres@csuchico.edu}

\subjclass[2010]{Primary: 05C25} 
\date{\today}

\begin{abstract}
We study the special value $u=1$ of Artin-Ihara $L$-functions associated to characters of the automorphism group of abelian covers of multigraphs.  In particular, we show an annihilation statement analogous to a classical conjecture of Brumer on annihilation of class groups for abelian extensions of number fields and we also calculate the index of an ideal analogous to the classical Stickelberger ideal in algebraic number theory.  Along the way, we make some observations about the number of spanning trees in abelian multigraph coverings that may be of independent interest.
\end{abstract} 
\maketitle 
\tableofcontents 

\section{Introduction}
The study of special values of $L$-functions in arithmetic geometry is a rich area of research that contains many conjectures and comparatively few unconditional results.  The equivariant Tamagawa number conjecture, as formulated by Burns and Flach in \cite{Burns/Flach:2001}, is an example of such a very general conjecture, which contains as special cases the more concrete Birch and Swinnerton-Dyer conjecture for elliptic curves (see \cite{Birch/Swinnerton-Dyer:1965}) and Stark's conjecture for Galois extensions of number fields (see \cite{Stark:1971}, \cite{Stark:1975}, \cite{Stark:1976}, and \cite{Stark:1980}).  Despite recent investigations by several authors, the latter two conjectures are still open in general.

Various problems and concepts from the theory of algebraic curves or from algebraic number theory have been transferred over to graph theory.  For instance, one can attach a zeta function to a multigraph, called the Ihara zeta function, and one can try to find analogues of the prime number theorem, the Riemann hypothesis and Siegel zeros for instance.  An analogue of Dirichlet's class number formula for Dedekind zeta functions has been found for the special value $u=1$ of Ihara zeta functions.  The invariant playing the role of the class number is the number of spanning trees of the multigraph which turns out to be the cardinality of a finite abelian group called the Jacobian of a multigraph.  Several of these questions have been studied in \cite{Terras:2011} for instance.

Now, starting with a Galois cover of multigraphs, the theory becomes an equivariant one, and one can define some $L$-functions attached to finite dimensional complex linear representations of the automorphism group of the cover.  These $L$-functions are analogous to the classical Artin $L$-functions, and we shall refer to them as Artin-Ihara $L$-functions.  As in algebraic number theory, the Ihara zeta function of the covering multigraph can be written as a product of finitely many Artin-Ihara $L$-functions, and it becomes natural to ask if some of the equivariant conjectures on special values of $L$-functions have analogues in the context of multigraphs.  In this paper, we shall show that there is an analogue to the classical Brumer conjecture on annihilation of class groups (see \cite{Rideout:1970}) for abelian covers of multigraphs and we shall also calculate the index of an ideal analogous to the classical Stickelberger ideal in algebraic number theory (see \cite{Sinnott:1981}).  Along the way, we make some observations about the number of spanning trees in abelian multigraph coverings that may be of independent interest. Specifically, in the case of an abelian cover of multigraphs, the number of trees in the top multigraph is divisible by that in the lower (see Corollary \ref{divisibility}). In Section \ref{relations} we see that, in the case of a 
$(\mathbb{Z}/2 \mathbb{Z})^m$ cover, the spanning tree number for the top multigraph is determined by those for the base and its intermediate double covers.

\section{Multigraphs}

We start by recalling what we mean by multigraphs.  
\begin{definition} \label{first_def}
\hfill
\begin{enumerate}
\item A multigraph $X$ consists of a set $V_{X}$ of vertices, and a multiset $E_{X}$ of unordered pairs of vertices whose elements are called edges.  If $e$ is an edge, then we let $V_{X}(e)$ denote the set consisting of the corresponding pair of vertices.  An edge $e$ is called a loop if $V_{X}(e)$ is a singleton.  An edge $e$ is said to be incident to a vertex $v$ if $v \in V_{X}(e)$.  If $v_{1}, v_{2} \in V_{X}$, then we write $v_{1} \sim v_{2}$ if there exists $e \in E_{X}$ such that $V_{X}(e) = \{v_{1},v_{2}\}$.
\item Let $X$ and $Y$ be multigraphs.  A morphism of multigraphs consists of two functions, which we denote by the same symbol, $f:V_{X} \longrightarrow V_{Y}$ and $f:E_{X} \longrightarrow E_{Y}$ satisfying $f(V_{X}(e)) = V_{Y}(f(e))$ for all $e \in E_{X}$.
\end{enumerate}
\end{definition}
For us, a graph will be a particular case of a multigraph.
\begin{definition}
\hfill
\begin{enumerate}
\item A graph is a multigraph with no loops and such that $V_{X}(e)$ is distinct for every edge $e \in E_{X}$.
\item If $X$ is a multigraph and $v$ is a vertex of $X$, then we denote by $d_{X}(v)$ the number of edges incident to $v$.  Here, a loop is counted twice.  The quantity $d_{X}(v)$ is called the valency (or the degree) of $v$.
\item A multigraph is said to be finite if both $V_{X}$ and $E_{X}$ are finite.  
\end{enumerate}
\end{definition}
Each edge of a multigraph $X$ can be given two different orientations in an obvious way.  If $e$ is an oriented edge of a graph, then it is clear what we mean by the initial and terminal vertices of $e$.  Also, if $e$ is an oriented edge of $X$, then we shall denote by $e^{-1}$ the same edge with opposite orientation.
\begin{definition}
\hfill
\begin{enumerate}
\item A path in $X$ consists of a sequence of oriented edges $e_{1} \cdot \ldots \cdot e_{m}$ such that the terminal vertex of $e_{i}$ is the same as the initial vertex of $e_{i+1}$ for all $i=1,\ldots,m-1$.  The initial vertex of $e_{1}$ is called the initial vertex of the path and the terminal vertex of $e_{m}$ is called the terminal vertex of the path.
\item A multigraph $X$ is connected if given any $v,v' \in V_{X}$, there is a path in $X$ going from $v$ to $v'$.
\item If $C = e_{1} \cdot \ldots \cdot e_{m}$ is a path, then its length is $m$ and is denoted by $\nu(C)$.
\item If $C = e_{1} \cdot \ldots \cdot e_{m}$ and $D = \varepsilon_{1}\cdot \ldots \cdot \varepsilon_{t}$ are paths such that the terminal vertex of $C$ is the same as the initial vertex of $D$, then we denote the path $e_{1} \cdot \ldots \cdot e_{m}\cdot \varepsilon_{1}\cdot \ldots \cdot \varepsilon_{t}$ by $C \cdot D$.
\item A path $e_{1} \cdot \ldots \cdot e_{m}$ is called closed if its initial vertex is the same as its terminal vertex.
\item A closed path $e_{1} \cdot \ldots \cdot e_{m}$ such that no edge is repeated independently of the orientation and such that the terminal vertex of $e_{i}$ is different than the terminal vertex of $e_{j}$ whenever $i \neq j$ is called a cycle.
\item A spanning tree of $X$ is a connected subgraph containing all of the vertices of $X$ and no cycles.  The number of spanning trees of $X$ is denoted by $\kappa_{X}$.
\end{enumerate}
\end{definition}

\emph{Throughout this paper, by a multigraph we will always mean a finite connected multigraph}.  From now on, we label the vertices $V_{X} = \{v_{1},\ldots, v_{n}\} $ so that $|V_{X}| = n$.  Some matrices attached to multigraphs are very important.
\begin{definition}
Let $X$ be a multigraph.
\begin{enumerate}
\item The adjacency matrix $A$ attached to $X$ is the $n \times n$ matrix $A=(a_{ij})$ defined via
\begin{equation*}
a_{ij} = 
\begin{cases}
\text{Twice the number of loops at the vertex }i, &\text{ if } i=j;\\
\text{The number of edges connecting the $i$th vertex to the $j$th vertex}, &\text{ if } i \neq j.
\end{cases}
\end{equation*}
\item The degree matrix $D$ attached to $X$ is the $n \times n$ diagonal matrix $D = (d_{ij})$ defined via $d_{ii} = d_{X}(v_{i})$.
\item The matrix $D - A$ is called the Laplacian matrix attached to $X$ and is denoted by $Q$.
\end{enumerate}
\end{definition}
All these matrices contain information about the corresponding multigraph $X$.  For example, we have the following important result sometimes known as Kirchhoff's theorem.
\begin{theorem} \label{Kirchhoff}
Let $X$ be a connected multigraph and $Q$ the corresponding Laplacian matrix.  Then ${\rm adj}(Q) = \kappa_{X} \cdot J$, where ${\rm adj}$ denotes the adjoint (or adjugate) of a matrix, and $J$ is the $n \times n$ matrix whose entries are all equal to $1$.
\end{theorem}
\begin{proof}
See Theorem $6.3$ on page $39$ of \cite{Biggs:1993} for example.  The author gives a proof for graphs only, but the proof can be adapted to multigraphs.
\end{proof}

\subsection{The Jacobian of a multigraph}
The following section is based on Part $1$ of \cite{Corry/Perkinson:2018} except that we allow loops in our multigraphs.  Recall that all our multigraphs are assumed to be finite and connected.

The divisor group on $X$ is defined to be the free abelian group on the vertices $V_{X}$.  It is an abelian group denoted by ${\rm Div}(X)$.  If $D = \sum_{v \in V_{X}}n_{v} \cdot v \in {\rm Div}(X)$, then we define
$${\rm deg}(D) = \sum_{v \in V_{X}} n_{v}.$$
This gives a group morphism ${\rm deg}:{\rm Div}(X) \longrightarrow \mathbb{Z}$ whose kernel will be denoted by ${\rm Div}^{\circ}(X)$.  We let $\mathcal{M}(X) = \mathbb{Z}^{V_{X}}$ be the set of $\mathbb{Z}$-valued functions on $V_{X}$.  For $v \in V_{X}$, we define $\chi_{v} \in \mathcal{M}(X)$ via
\begin{equation*}
\chi_{v}(v_{0}) = 
\begin{cases}
1, &\text{ if } v = v_{0};\\
0, &\text{ if } v \neq v_{0}.
\end{cases}
\end{equation*}
The functions $\chi_{v}$, as $v$ runs over $V_{X}$, form a $\mathbb{Z}$-basis for $\mathcal{M}(X)$.  One then defines a group morphism ${\rm div}:\mathcal{M}(X) \longrightarrow {\rm Div}(X)$ on the basis elements $\chi_{v}$ (and extending by $\mathbb{Z}$-linearity) via 
$$\chi_{v_{0}} \mapsto {\rm div}(\chi_{v_{0}}) = \sum_{v \in V_{X}} \rho_{v}(v_{0}) \cdot v, $$
where
\begin{equation*}
\rho_{v}(v_{0}) = 
\begin{cases}
 d_{X}(v_{0}) - 2\cdot \text{number of loops at $v_{0}$}, &\text{ if } v = v_{0}; \\
-\text{number of edges from $v$ to $v_{0}$}, &\text{ if } v \neq v_{0}.
\end{cases}
\end{equation*}
The reader will notice that the map ${\rm div}$ is given by the Laplacian matrix $Q = D - A$ after an appropriate choice of $\mathbb{Z}$-bases for $\mathcal{M}(X)$ and ${\rm Div}(X)$.  We let ${\rm Pr}(X) = {\rm div}(\mathcal{M}(X))$ which is a subgroup of ${\rm Div}(X)$ and furthermore, we let ${\rm Pic}(X) = {\rm Div}(X)/{\rm Pr}(X)$. Note that ${\rm Pr}(X) \subseteq {\rm Div}^{\circ}(X)$, and thus one also defines ${\rm Jac}(X) = {\rm Div}^{0}(X)/{\rm Pr}(X)$.
\begin{theorem}
Let $X$ be a connected multigraph.  Then ${\rm Jac}(X)$ is a finite group and furthermore $|{\rm Jac}(X)| = \kappa_{X}$.
\end{theorem}
\begin{proof}
See Remark $2.38$ on page $35$ of  \cite{Corry/Perkinson:2018}.
\end{proof}

\subsection{Ihara zeta functions}
\emph{From now on, not only do we assume that our multigraphs are finite and connected, but we also assume that they do not contain vertices of degree one}. 

\begin{definition}
Let $X$ be a multigraph.  
\begin{enumerate}
\item A path $e_{1} \cdot \ldots \cdot e_{m}$ has a backtrack if $e_{i+1} = e_{i}^{-1}$ for some $i \in \{1,\ldots,m-1\}$.
\item A path $e_{1} \cdot \ldots \cdot e_{m}$ has a tail if $e_{1} = e_{m}^{-1}$.
\item A closed path is called a prime path if 
\begin{enumerate}
\item It has no backtracks,
\item It has no tails,
\item It is not of the form $C^{a}$ for some closed path $C$ and some integer $a \ge 2$.
\end{enumerate}
\end{enumerate}
\end{definition}

One defines an equivalence relation on closed paths as follows.
\begin{definition} \label{equiv}
Two closed paths $e_{1}\cdot \ldots \cdot e_{m}$ and $\varepsilon_{1} \cdot \ldots \cdot \varepsilon_{t}$ are called equivalent if $m = t$ and if there exists an $m$-cycle $\sigma$ such that $e_{i} = \varepsilon_{\sigma(i)}$ for $i=1,\ldots,m$.
\end{definition}

The following definition is essential for the definition of the Ihara zeta function of a multigraph.
\begin{definition}
Let $X$ be a multigraph.  A prime in $X$ is an equivalence class of prime paths for the equivalence relation of Definition \ref{equiv}.
\end{definition}
We shall typically denote a prime by $\mathfrak{p}$.  The Ihara zeta function of a multigraph $X$ is defined to be
$$\zeta_{X}(u) = \prod_{\mathfrak{p}}(1-u^{\nu(\mathfrak{p})})^{-1}, $$
where the product is over all primes $\mathfrak{p}$ of $X$.  This product is usually infinite and converges if $|u|$ is small enough.  Interestingly, the Ihara zeta function of a multigraph is the reciprocal of a polynomial in $u$ and this shows that it can be extended to a meromorphic function on $\mathbb{C}$ (in fact a rational function):

\begin{theorem}[Three-term determinant formula] \label{three_term}
Let $X$ be a multigraph, $A$ the adjacency matrix and $D$ the degree matrix of $X$.  Let also $r = |E_{X}| - |V_{X}| + 1$.  Then, we have
$$\frac{1}{\zeta_{X}(u)} = (1-u^{2})^{r-1} \cdot {\rm det}(I - Au + (D - I)u^{2}). $$
\end{theorem}
\begin{proof}
We refer the reader to Theorem $2.5$ on page $17$ of \cite{Terras:2011}.  
\end{proof}
This last theorem is quite convenient.  For example, one can prove the following analogue to Dirichlet's class number formula in algebraic number theory which is an exercise on page $78$ of \cite{Terras:2011}, but we present a proof for the reader's convenience.

\begin{theorem} \label{analogue_dir}
Let $X$ be a multigraph, and assume that $r= |E_{X}| - |V_{X}| + 1 \neq 1$.  Then, one has
$${\rm ord}_{u=1}\left(\zeta_{X}(u)^{-1}\right) = r \text{ and } \zeta_{X}^{*}(1) = (-1)^{r+1} \cdot 2^{r} \cdot (r-1) \cdot \kappa_{X},$$
where $\zeta_{X}^{*}(1)$ denotes the first non-vanishing Taylor coefficent of $\zeta_{X}(u)^{-1}$ at $u=1$.
\end{theorem}
\begin{proof}
Let $g(u) = (1-u^{2})^{r-1}$ and let $h(u) = {\rm det}(I - Au + (D - I)u^{2})$.  Both $g$ and $h$ are polynomials in $u$.  By Leibniz's formula, one has
$$\frac{d^{r}}{du^{r}} \zeta_{X}(u)^{-1} = \sum_{i=0}^{r} \binom{r}{i}g^{(r-i)}(u) h^{(i)}(u).$$
Since $u=1$ is a zero of order $r-1$ for the polynomial $g$, we have
$$\frac{d^{r}}{du^{r}}\zeta_{X}(u)^{-1}\Big|_{u=1} = rg^{(r-1)}(1)h'(1) + g^{(r)}(1)h(1). $$
Note that $h(1) = {\rm det}(I - A + D - I) = {\rm det}(Q) = 0$.  Therefore, we get
$$\frac{d^{r}}{du^{r}}\zeta_{X}(u)^{-1}\Big|_{u=1} = rg^{(r-1)}(1)h'(1). $$
Now, a simple calculation shows that $g^{(r-1)}(1) = (-2)^{r-1} \cdot (r-1)!$, and thus we are left to show $h'(1) = 2 (r-1) \kappa_{X}$ in order to prove our claim.  Using Jacobi's formula, we calculate
\begin{equation*}
\begin{aligned}
\frac{d}{du}h(u) &= \frac{d}{du} {\rm det}(I - Au + (D - I)u^{2}) \\
&={\rm tr}\left({\rm adj}(I - Au + (D-I)u^{2})  \cdot \frac{d}{du} \left(I - Au + (D - I)u^{2} \right) \right) \\
&= {\rm tr}\left({\rm adj}(I - Au + (D-I)u^{2}) \cdot \left(-A + 2 (D - I)u \right) \right),
\end{aligned}
\end{equation*}
where ${\rm adj}$ denotes the adjoint of a matrix.  Thus, 
\begin{equation*}
\begin{aligned}
\frac{d}{du}h(u)\Big|_{u=1} &= {\rm tr}\left({\rm adj}(Q) (-A + 2(D - I)) \right) \\
&= {\rm tr}(\kappa_{X} \cdot J (-A +2(D - I))),
\end{aligned}
\end{equation*}
by Theorem \ref{Kirchhoff}.  Therefore, one has 
$$\frac{d}{du}h(u)\Big|_{u=1} = \kappa_{X} \cdot \left( {\rm tr}(JQ) + {\rm tr}(JD) + {\rm tr}(-2J) \right),$$
but we have
\begin{enumerate}
\item ${\rm tr}(JQ) = 0$,
\item ${\rm tr}(JD) = 2 \cdot |E_{X}|$ by the degree sum theorem,
\item ${\rm tr}(-2J) = -2 \cdot |V_{X}|$,
\end{enumerate}
and therefore we get
$$h'(1) = (2 |E_{X}|-2|V_{X}|) \cdot \kappa_{X} =  2(r-1) \kappa_{X},$$
as we wanted to show.
\end{proof}
\begin{remark} \label{cycle}
Let $X$ be a connected multigraph with no degree one vertices satisfying $r=1$.  Then $X$ is necessarily the cycle graph on $n$ vertices which we denote by $C_{n}$.  In this case, there are only two primes and $(\zeta_{X}(u))^{-1} = (1-u^{n})^{2}$.  It follows that 
$${\rm ord}_{u=1}(\zeta_{X}(u)^{-1}) = 2 \text{ and } \zeta_{X}^{*}(1) = 2  n^{2}. $$
\end{remark}

\section{Galois covers of multigraphs}
We now view a multigraph as a finite $CW$-complex of dimension one and as such we view our multigraphs as topological spaces.  One can then talk about covering spaces and Galois (or regular) covering spaces of multigraphs.  See for instance Chapter $6$ of \cite{Massey:1967}.  Note that a covering map of multigraphs is necessarily a morphism of multigraphs as previously defined in Definition \ref{first_def}, and an automorphism of a cover of multigraphs is necessarily an isomorphism of multigraphs.  Such a covering map is a $d$-to-$1$ function for some integer $d \ge 1$ which we will refer to as the degree of the cover.  (We point out that the word degree means two different things in this paper:  the degree of a divisor and the degree of a cover.  We hope that this will not cause any confusion.)  

Let $Y/X$ be a covering of a graph with projection map $\pi$ and let $\mathfrak{P}$ be a prime of $Y$.  If $\mathfrak{P} = [P]$ for some prime path $P$, then $\pi(P) = p^{f}$ for some prime path $p$ of $X$ and some positive integer $f$.  Let $\mathfrak{p}$ be the equivalence class of $p$.  Then one says that $\mathfrak{P}$ lies above $\mathfrak{p}$ and $f$ is called the residual degree of $\mathfrak{P}$ over $\mathfrak{p}$.  

If we assume furthermore that the cover is Galois of degree $d$, say, and $r$ is the number of primes of $Y$ lying above $\mathfrak{p}$, one can show that $f \cdot r = d$.  (See part $(6)$ of Theorem $16.5$ on page $137$ of \cite{Terras:2011}.)  Again, assume that we have a prime $\mathfrak{P}$ of $Y$ lying above a prime $\mathfrak{p}$ of $X$.  Assume that $\mathfrak{P}$ corresponds to a prime path starting at $w$.  Let $v = \pi(w)$ and let $\sigma$ be a prime path in $X$ corresponding to $\mathfrak{p}$ and starting at $v$.  There is a unique lift $\tilde{\sigma}$ of $\sigma$ to $Y$ with initial vertex $w$.  Let $w'$ be the terminal vertex of $\tilde{\sigma}$.  One defines the Frobenius automorphism of $\mathfrak{P}$ over $\mathfrak{p}$, denoted by $\left(\frac{Y/X}{\mathfrak{P}} \right)$ to be the unique automorphism $g \in {\rm Aut}(Y/X)$ such that $g \cdot w = w'$.  One can show that this definition does not depend on any of the choices made above.  If $\mathfrak{P}_{1}$ and $\mathfrak{P}_{2}$ are two primes of $Y$ lying above the same prime $\mathfrak{p}$, then there exists $\sigma \in {\rm Aut}(Y/X)$ such that $\mathfrak{P}_{1}^{\sigma} =\mathfrak{P}_{2}$ and
$$\left(\frac{Y/X}{\mathfrak{P}_{2}} \right)= \sigma \cdot \left(\frac{Y/X}{\mathfrak{P}_{1}} \right) \cdot \sigma^{-1}.$$
Thus, if the cover has an abelian automorphism group, then the Frobenius automorphism depends only on $\mathfrak{p}$ and will be denoted by $\left(\frac{Y/X}{\mathfrak{p}} \right)$.

\subsection{Artin-Ihara $L$-functions} \label{artin_ihara_L}
Let $Y/X$ be a Galois cover of multigraphs.  We will denote ${\rm Aut}(Y/X)$ simply by $G$ and we shall assume from now on that $G$ is abelian.  We will refer to such a cover as an abelian cover.  We denote the group of characters of $G$ by $\widehat{G}$.  If $\chi \in \widehat{G}$, then the Artin-Ihara $L$-function is defined by the formal infinite product
$$L_{Y/X}(u,\chi) = \prod_{\mathfrak{p}} \left(1 - \chi\left(\left(\frac{Y/X}{\mathfrak{p}}\right) \right)u^{\nu(\mathfrak{p})} \right)^{-1}, $$
where the product is over all primes in $X$.  As for the Ihara zeta function, this product is usually infinite and it can be shown to converge when $|u|$ is small enough.  From now on, we let $\chi_{1}$ be the trivial character of the group $G$.  Note that $L_{Y/X}(u,\chi_{1}) = \zeta_{X}(u)$.
\begin{theorem}  \label{product_form}
Let $Y/X$ be an abelian cover of multigraphs.  Then one has
$$\zeta_{Y}(u) = \zeta_{X}(u) \cdot \prod_{\chi \neq \chi_{1}} L_{Y/X}(u,\chi). $$
\end{theorem}
\begin{proof}
See Corollary $18.11$ in \cite{Terras:2011}.
\end{proof}

If we write 
$$L_{Y/X}(u,\chi)^{-1} = c(\chi) (u-1)^{r(\chi)} + \ldots $$
with $c(\chi) \neq 0$, then it becomes a natural question to ask what $c(\chi)$ and $r(\chi)$ are.  The number $c(\chi_{1})$ is in fact $\zeta_{X}^{*}(1)$ of Theorem \ref{analogue_dir} and Remark \ref{cycle}.  We shall also use $L_{Y/X}^{*}(1,\chi)$ to denote $c(\chi)$.  

From now on, we fix a vertex $w_{i}$ of $Y$ in the fiber of $v_{i}$ for each $i=1,\ldots,n$.
\begin{definition}
Let $Y/X$ be an abelian cover of multigraphs with automorphism group $G$.
\begin{enumerate}
\item For $\sigma \in G$, we define the matrix $A(\sigma)$ to be the $n \times n$ matrix $A(\sigma)=(a_{ij}(\sigma))$ defined via
\begin{equation*}
a_{ij}(\sigma) = 
\begin{cases}
\text{Twice the number of loops at the vertex }w_{i}, &\text{ if } i=j \text{ and } \sigma = 1;\\
\text{The number of edges connecting $w_{i}$ to $w_{j}^{\sigma}$}, &\text{ otherwise}.
\end{cases}
\end{equation*}
\item If $\chi \in \widehat{G}$, then we let
$$A_{\chi} = \sum_{\sigma \in G} \chi(\sigma) \cdot A(\sigma). $$
\end{enumerate}
\end{definition}

The following theorem is again very convenient.

\begin{theorem}[Three-term determinant formula for $L$-functions] \label{three_term_L}
Let $Y/X$ be an abelian cover of multigraphs with automorphism group $G$ and let $\chi \in \widehat{G}$.  Then, we have
$$\frac{1}{L_{Y/X}(u,\chi)} = (1-u^{2})^{r_{X}-1} \cdot {\rm det}(I - A_{\chi}u + (D - I)u^{2}), $$
where again $r_{X} = |E_{X}| - |V_{X}| + 1$.
\end{theorem}
\begin{proof}
We refer the reader to Theorem $18.15$ on page $156$ of \cite{Terras:2011}.  
\end{proof}

This last theorem allows us to prove:
\begin{proposition} \label{order_of_vanishing}
Let $Y/X$ be an abelian cover of multigraphs with Galois group $G$.  If $\chi$ is a non-trivial character, one has $r(\chi) = r_{X} -1$.
\end{proposition}
\begin{proof}
Indeed, the decomposition
$$\zeta_{Y}(u) = \zeta_{X}(u) \cdot \prod_{\chi \neq \chi_{1}} L_{Y/X}(u,\chi) $$
of Theorem \ref{product_form} gives
\begin{equation} \label{order_relation}
r_{Y} = r_{X} + \sum_{\chi \neq \chi_{1}} r(\chi). 
\end{equation}
Now, because of Theorem \ref{three_term_L}, we have
\begin{equation} \label{lower_bound}
r(\chi) \ge r_{X} - 1 
\end{equation}
for non-trivial characters $\chi \in \widehat{G}$.  If $|G|=d$, then $|E_{Y}| = d|E_{X}|$ and similarly $|V_{Y}| = d|V_{X}|$.  A simple calculation then shows that
\begin{equation} \label{count}
r_{Y} = r_{X} + (d-1)(r_{X} - 1). 
\end{equation}
If $r(\chi)> r_{X} -1$ for some non-trivial character $\chi$, then it would follow from $(\ref{order_relation})$, $(\ref{lower_bound})$, and $(\ref{count})$ that 
$$r_{Y} = r_{X} + \sum_{\chi \neq \chi_{1}} r(\chi) > r_{X} + (d-1)(r_{X} - 1)= r_{Y},$$ 
but this is a contradiction.  Thus $r(\chi) = r_{X} - 1$ for all $\chi \in \widehat{G}$ satisfying $\chi \neq \chi_{1}$.
\end{proof}
In summary, one has
\begin{equation*}
r(\chi) = r_{X}-1 \text{ if } \chi \neq \chi_{1}, \text{ and } r(\chi_{1}) = 
\begin{cases}
r_{X}, & \text{ if } X \neq C_{n};\\
2, & \text{ if } X = C_{n}.
\end{cases}
\end{equation*}
\begin{corollary} \label{special_value_one}
Let $Y/X$ be an abelian cover with automorphism group $G$ and let $\chi \in \widehat{G}$.  If $\chi \neq \chi_{1}$, then
$$L_{Y/X}^{*}(1,\chi) = (-2)^{r_{X}-1} \cdot {\rm det}(D - A_{\chi}). $$
\end{corollary}
\begin{proof}
Indeed, Proposition \ref{order_of_vanishing} combined with Theorem \ref{three_term_L} shows that the polynomial $h(u) = {\rm det}(I - A_{\chi}u + (D-I)u^{2})$ does not vanish at $u=1$.  The result follows then from Theorem \ref{three_term_L}, for one just has to calculate the $(r_{X}-1)$-th derivative of $L_{Y/X}(u,\chi)^{-1}$ at $u=1$, which is a simple calculation left to the reader.
\end{proof}

\subsection{Relations between the number of spanning trees in abelian covers} \label{relations}
Artin-Ihara $L$-functions obey the same formalism as the usual Artin $L$-functions in number theory.  (See Proposition $18.10$ of \cite{Terras:2011}).  This allows us to show the following theorem which is analogous to Kuroda's class number formula for biquadratic extensions of number fields.  (See \cite{Lemmermeyer:1994} for instance.)

\begin{theorem}
Let $Y/X$ be an abelian cover of multigraphs with automorphism group $G \simeq \mathbb{Z}/2\mathbb{Z} \times \mathbb{Z}/2\mathbb{Z}$.  Let $X_{i}$ be the intermediate double covers of $X$ for $i=2,3,4$.  Then
$$\kappa_{Y} = 2 \cdot \frac{\kappa_{2} \kappa_{3} \kappa_{4}}{\kappa_{X}^{2}},$$
where we write $\kappa_{i}$ for $\kappa_{X_{i}}$ ($i=2,3,4$) in order to simplify the notation.
\end{theorem}
\begin{proof}
We have four characters of $G$ which we label $\widehat{G}= \{\chi_{1},\chi_{2},\chi_{3},\chi_{4}\}$, so that $\chi_{1}$ is the trivial character and $\chi_{i}$, for $i=2,3,4$, is the character satisfying ${\rm ker}(\chi_{i}) = {\rm Aut}(Y/X_{i})$.  From Theorem \ref{product_form}, we have
\begin{equation} \label{un}
\zeta_{Y}^{*}(1) = \zeta_{X}^{*}(1) \cdot \prod_{i=2}^{4} L_{Y/X}^{*}(1,\chi_{i}). 
\end{equation}
Now, $\chi_{i}$ induces the unique non-trivial character $\widetilde{\chi}_{i}$ of ${\rm Aut}(X_{i}/X)$ for $i=2,3,4$.  The inflation property of Artin-Ihara $L$-functions (point $(2)$ of Theorem $18.10$ of \cite{Terras:2011}) shows that
\begin{equation} \label{deux}
L_{Y/X}(u,\chi_{i}) = L_{X_{i}/X}(u,\widetilde{\chi}_{i}), 
\end{equation}
for $i=2,3,4$.  Applying Theorem \ref{product_form} to the cover $X_{i}/X$ gives
\begin{equation} \label{trois}
L_{X_{i}/X}(u,\widetilde{\chi}_{i}) = \frac{\zeta_{X_{i}}(u)}{\zeta_{X}(u)}. 
\end{equation}
Combining (\ref{un}), (\ref{deux}), and (\ref{trois}) together, we obtain
$$\zeta_{Y}^{*}(1) = \zeta_{X}^{*}(1) \cdot \prod_{i=2}^{4}\frac{\zeta_{X_{i}}^{*}(1)}{\zeta_{X}^{*}(1)}. $$
Using Theorem \ref{analogue_dir}, a simple calculation shows the equality
$$\kappa_{Y} = 2 \cdot \frac{\kappa_{2} \kappa_{3} \kappa_{4}}{\kappa_{X}^{2}},$$
where we write $\kappa_{i}$ for $\kappa_{X_{i}}$ ($i=2,3,4$) as we wanted to show.
\end{proof}
\begin{remark}
This last theorem can be extended to abelian covers with Galois group isomorphic to an elementary abelian $2$-group $(\mathbb{Z}/2\mathbb{Z})^{m}$.  Indeed, if $Y/X$ is such a cover, and $X_{i}$ are the intermediate double covers for $i=1,\ldots,2^{m}-1$, then one can show 
$$\kappa_{Y} = \frac{2^{2^{m}-m-1}}{\kappa_{X}^{2^{m} - 2}}\prod_{i=1}^{2^{m}-1}\kappa_{i}. $$
\end{remark}
This type of relation between various numbers of spanning trees could be generalized to other abelian (and more generally Galois) covers of multigraphs.  It would be interesting to find the most general one possible perhaps along the lines of Brauer's class number relation in algebraic number theory.

\section{The equivariant special value}
Again, we assume that $Y/X$ is an abelian cover of multigraphs with Galois group $G$.  We introduce an equivariant $L$-function $\theta_{Y/X}:\mathbb{C} \longrightarrow \mathbb{C}[G]$ defined by
$$u \mapsto \theta_{Y/X}(u) = \sum_{\chi \in \widehat{G}} L_{Y/X}(u,\chi)^{-1} \cdot e_{\overline{\chi}},$$
where
$$e_{\chi} = \frac{1}{|G|} \sum_{\sigma \in G} \chi(\sigma) \cdot \sigma^{-1} $$
is the usual idempotent in $\mathbb{C}[G]$ corresponding to the character $\chi$.  We are interested in the special value
$$\theta_{Y/X}^{*}(1) = \sum_{\chi \in \widehat{G}} L_{Y/X}^{*}(1,\chi) \cdot e_{\overline{\chi}}. $$
As before, we label the vertices $V_{X} = \{v_{1},\ldots,v_{n} \}$, and recall that in \S \ref{artin_ihara_L}, we fixed a vertex $w_{i}$ of $Y$ in the fiber of $v_{i}$ for each $i=1,\ldots,n$.  Then one has
$${\rm Div}(Y) = \bigoplus_{i=1}^{n} \mathbb{Z}[G] \cdot w_{i}. $$
Since, ${\rm Ann}_{\mathbb{Z}[G]}(w_{i}) = 0$, one has $\mathbb{Z}[G] \cdot w_{i} \simeq \mathbb{Z}[G]$ as $\mathbb{Z}[G]$-modules and therefore, ${\rm Div}(Y)$ is a free $\mathbb{Z}[G]$-module of rank $n$.  We consider the function $\phi:{\rm Div}(Y) \longrightarrow {\rm Div}(Y)$ given on vertices $w_{0} \in V_{Y}$ by the formula
$$w_{0} \mapsto \phi(w_{0}) = \sum_{w \in V_{Y}} \rho_{w}(w_{0}) \cdot w,$$
where
\begin{equation*}
\rho_{w}(w_{0}) =
\begin{cases}
d_{Y}(w_{0}) - 2 \cdot \text{number of loops at $w_{0}$}, &\text{if } w = w_{0};\\
- \text{number of edges from } w \text{ to } w_{0}, &\text{if } w \neq w_{0}.
\end{cases}
\end{equation*}
Note that ${\rm Div}(Y)$ is isomorphic to $\mathcal{M}(Y)$ as $\mathbb{Z}[G]$-modules via the map $w \mapsto \chi_{w}$, which leads to the following commutative diagram
\begin{equation*} 
  \begin{CD}
     {\rm Div}(Y)       @>{\phi}>>          {\rm Div}(Y)             \\
     @V{\rotatebox{90}{$\simeq$}}VV  @V{\rotatebox{90}{$=$}}VV   \\
     \mathcal{M}(Y)    @>{{\rm div}}>>       {\rm Div}(Y)
  \end{CD}
\end{equation*}

\begin{lemma} \label{usefullem}
For all $w_{1},w_{2} \in V_{Y}$ and for all $\sigma \in G$, one has
$$\rho_{w_{1}}(\sigma \cdot w_{2}) = \rho_{\sigma^{-1} \cdot w_{1}}(w_{2}). $$
\end{lemma}
\begin{proof}
Since $\sigma$ is an isomorphism of multigraphs, it induces two bijections which we denote by the same symbol $\sigma:V_{Y} \longrightarrow V_{Y}$ and $\sigma:E_{Y} \longrightarrow E_{Y}$ satisfying $\sigma(V_{Y}(e)) = V_{Y}(\sigma(e))$ for all $e \in E_{Y}$.  Thus, one has $w_{1} \sim w_{2}$ if and only if $\sigma \cdot w_{1} \sim \sigma \cdot w_{2}$ for all $\sigma \in G$.  It follows that $\rho_{w_{1}}(\sigma \cdot w_{2}) = \rho_{\sigma^{-1} \cdot w_{1}}(w_{2})$ as we wanted to show.
\end{proof}
\begin{corollary}
The group morphism $\phi:{\rm Div}(Y) \longrightarrow {\rm Div}(Y)$ is a morphism of $\mathbb{Z}[G]$-modules.
\end{corollary}
\begin{proof}
Using Lemma \ref{usefullem}, if $\sigma \in G$, one has
\begin{equation*}
\begin{aligned}
\phi(\sigma \cdot w_{0}) &= \sum_{w \in V_{Y}} \rho_{w}(\sigma \cdot w_{0}) \cdot w \\
&= \sum_{w \in V_{Y}} \rho_{\sigma^{-1} \cdot w}(w_{0}) \cdot w \\
&= \sum_{w \in V_{Y}} \rho_{w}(w_{0})\cdot (\sigma \cdot w) \\
&= \sigma \cdot \phi(w_{0}).
\end{aligned}
\end{equation*}
\end{proof}

\begin{definition}
For $i=1,\ldots,n$, we define $\ell_{w_{i}}:{\rm Div}(Y) \longrightarrow \mathbb{Z}[G]$ via the formula
$$w \mapsto \ell_{w_{i}}(w) = \sum_{\sigma \in G} \rho_{w_{i}}(\sigma \cdot w) \cdot \sigma^{-1}. $$
\end{definition}
We will often write $\ell_{i}$ instead of $\ell_{w_{i}}$ in order to simplify the notation.  It is simple to check that the maps $\ell_{i}$ are morphisms of $\mathbb{Z}[G]$-modules.
\begin{proposition}
With the notation as above, one has
$$\phi(D) = \sum_{i=1}^{n} \ell_{i}(D) \cdot w_{i}, $$
for any $D \in {\rm Div}(Y)$.
\end{proposition}
\begin{proof}
It suffices to show this equality for $D = w$, where $w \in V_{Y}$.  Using Lemma \ref{usefullem} and the fact that $G$ acts transitively on the fibers of each $v_{i}$, one calculates
\begin{equation*}
\begin{aligned}
\sum_{i=1}^{n} \ell_{i}(w) \cdot w_{i} &= \sum_{i=1}^{n} \sum_{\sigma \in G} \rho_{w_{i}}(\sigma \cdot w) \cdot (\sigma^{-1} \cdot w_{i}) \\
&= \sum_{i=1}^{n} \sum_{\sigma \in G} \rho_{\sigma^{-1} \cdot w_{i}}(w) \cdot (\sigma^{-1}\cdot w_{i}) \\
&= \phi(w).
\end{aligned}
\end{equation*}
\end{proof}

The morphism of $\mathbb{Z}[G]$-modules
$$\wedge^{n}\phi: \bigwedge^{n}_{\mathbb{Z}[G]} {\rm Div}(Y) \longrightarrow \bigwedge^{n}_{\mathbb{Z}[G]}{\rm Div}(Y), $$
has the property
$$\wedge^{n}\phi(w_{1} \wedge \ldots \wedge w_{n}) = {\rm det}_{\mathbb{Z}[G]}(\phi) \cdot w_{1} \wedge \ldots \wedge w_{n}, $$
and thus
\begin{equation*}
{\rm det}_{\mathbb{Z}[G]}(\phi) = {\rm det}_{\mathbb{Z}[G]} (\ell_{i}(w_{j})) = {\rm det}_{\mathbb{Z}[G]}\left(\sum_{\sigma \in G} \rho_{w_{i}}(\sigma \cdot w_{j}) \cdot \sigma^{-1} \right).
\end{equation*}

\begin{theorem} \label{equivariant_special_value}
With the notation as above, one has
$$(-2)^{r_{X} - 1} \cdot {\rm det}_{\mathbb{Z}[G]}(\phi) = \theta_{Y/X}^{*}(1) \cdot e, $$
where $e = 1 - e_{\chi_{1}}$.
\end{theorem}
\begin{proof}
Let $\chi \in \widehat{G}$. One has
\begin{equation*}
\overline{\chi(\sigma)} \cdot (\rho_{w_{i}}(\sigma \cdot w_{j})) =
\begin{cases}
- \overline{\chi(\sigma)} \cdot A(\sigma), &\text{ if } \sigma \neq 1;\\
D - A(1), &\text{ if } \sigma =1.
\end{cases}
\end{equation*}
We then have the following equality of matrices
$$\left(\sum_{\sigma \in G} \overline{\chi(\sigma)} \rho_{w_{i}}(\sigma \cdot w_{j}) \right) = D - A_{\overline{\chi}}. $$
It follows that $\chi \left( (-2)^{r_{X} - 1} \cdot {\rm det}_{\mathbb{Z}[G]}(\phi)\right) = (-2)^{r_{X}-1} \cdot {\rm det}(D - A_{\overline{\chi}})$.  But by Corollary \ref{special_value_one}, we have
\begin{equation*}
\chi \left( (-2)^{r_{X} - 1} \cdot {\rm det}_{\mathbb{Z}[G]}(\phi)\right) =
\begin{cases}
L_{Y/X}^{*}(1,\overline{\chi}), &\text{ if } \chi \neq \chi_{1}; \\
0, &\text{ if } \chi = \chi_{1}.
\end{cases}
\end{equation*}
This ends the proof of the theorem.
\end{proof}
\begin{remark}
We can now explain what would happen if we had chosen other vertices $w_{i}'$ in the fiber of $v_{i}$ for $i=1,\ldots,n$.  The automorphism group $G$ acts transitively on the fibers, and thus there exist $\tau_{i} \in G$ such that $\tau_{i} \cdot w_{i} = w_{i}'$.  If we let $P$ be the diagonal matrix whose elements on the diagonal are $\tau_{1},\ldots,\tau_{n}$, then $P \in {\rm Gl}(n,\mathbb{Z}[G])$.  Furthermore, one can check that $P^{-1} \cdot (\ell_{w_{i}}(w_{j})) \cdot P = (\ell_{w_{i}'}(w_{j}'))$ and
$$I - A_{\chi}'u + (D-I)u^{2} = \chi(P)(I-A_{\chi}u + (D-I)u^{2})\chi(P^{-1}), $$
where $A_{\chi}'$ is the matrix obtained by using the vertices $w_{i}'$.  It follows from Theorem \ref{three_term_L} that the $L$-functions, and hence also the special values $L_{Y/X}^{*}(1,\chi)$, do not depend on the choice of the vertices $w_{i}$.
\end{remark}

\subsection{An annihilation statement}
Note that it follows from Theorem \ref{equivariant_special_value} that $\theta_{Y/X}^{*}(1) \cdot e \in \mathbb{Z}[G]$.  We can now prove the following analogue of a classical conjecture of Brumer on annihilation of class groups.
\begin{theorem}
Let $Y/X$ be an abelian cover of multigraphs with automorphism group $G$.  Then, we have
$$\theta_{Y/X}^{*}(1) \cdot e \in {\rm Ann}_{\mathbb{Z}[G]}({\rm Jac}(Y)), $$
where again $e = 1 - e_{\chi_{1}}$.
\end{theorem}
\begin{proof}
In fact, we shall show that $\theta_{Y/X}^{*}(1) \cdot e \in {\rm Ann}_{\mathbb{Z}[G]}({\rm Pic}(Y))$.  Indeed, let $D \in {\rm Div}(Y)$.  Then one has $\phi \circ \phi^{{\rm adj}} = {\rm det}_{\mathbb{Z}[G]}(\phi) \cdot {\rm id}_{{\rm Div}(Y)}$, where $\phi^{{\rm adj}}$ denotes the adjoint (or adjugate) of $\phi$.  Therefore, by Theorem \ref{equivariant_special_value}, we have
\begin{equation*}
\begin{aligned}
\theta_{Y/X}^{*}(1) \cdot e \cdot D &= (-2)^{r_{X}-1} \cdot {\rm det}_{\mathbb{Z}[G]}(\phi) \cdot D \\
&= (-2)^{r_{X}-1} \cdot \phi \circ \phi^{{\rm adj}}(D) \\
&= \phi \left((-2)^{r_{X}-1} \cdot \phi^{{\rm adj}}(D) \right) \in {\rm Pr}(Y),
\end{aligned}
\end{equation*}
and this is precisely what we wanted to show.
\end{proof}
We remark that we have actually showed the inclusion
$$\theta_{Y/X}^{*}(1) \cdot e \in {\rm Ann}_{\mathbb{Z}[G]}({\rm Pic}(Y)). $$
This phenomenon also happens in the function field case if the cardinality of the auxiliary set of primes $S$ satisfies $|S| > 1$.  See for instance the statement of the Brumer-Stark conjecture on page $267$ of \cite{Rosen:2002}.

\subsection{The index of the ideal generated by the special value}
We have a natural $\mathbb{Z}[G]$-module morphism $s:\mathbb{Z}[G] \longrightarrow \mathbb{Z}$, defined by $\sigma \mapsto s(\sigma) = 1$, where $\sigma \in G$, and where $G$ acts trivially on $\mathbb{Z}$.  The kernel of this morphism is the augmentation ideal and is denoted by $I_{G}$.  We then have a short exact sequence of $\mathbb{Z}[G]$-modules
$$0 \longrightarrow I_{G} \longrightarrow \mathbb{Z}[G] \stackrel{s}{\longrightarrow} \mathbb{Z} \longrightarrow 0. $$
Note that $\theta_{Y/X}^{*}(1) \cdot e \in I_{G}$.
\begin{theorem} \label{index}
Let $Y/X$ be an abelian cover of degree $d$ with automorphism group $G$ and let $e = 1 - e_{\chi_{1}}$.  We have
$$|I_{G}/\theta_{Y/X}^{*}(1)\cdot e \cdot \mathbb{Z}[G]| =  2^{(d-1)(r_{X} - 1)} \cdot \frac{\kappa_{Y}}{\kappa_{X}}.$$
\end{theorem}
\begin{proof}
Consider the $\mathbb{Z}[G]$-module morphism $T:\mathbb{Z}[G] \longrightarrow \mathbb{Z}[G]$ given by multiplication by the element $\theta_{Y/X}^{*}(1) \cdot e$.  Since $\theta_{Y/X}^{*}(1) \cdot e \in (\mathbb{Q}[G]\cdot e)^{\times}$, we have ${\rm ker}(T_{\mathbb{Q}}) = \mathbb{Q} \cdot e_{\chi_{1}}$.  Therefore ${\rm ker}(T) = {\rm ker}(T_{\mathbb{Q}}) \cap \mathbb{Z}[G] = \mathbb{Z} \cdot N_{G}$.  The map $T$ leads to the following commutative diagram whose rows are exact:
\begin{equation*} 
	\begin{CD}
		 0    @>>>   I_{G}      @>>>          \mathbb{Z}[G]       @>{s}>> \mathbb{Z}  @>>>0 \\
		 & &  @V{T}VV   @V{T}VV         @VVV    &  \\
		 0    @>>>   I_{G}      @>>>          \mathbb{Z}[G]      @>{s}>> \mathbb{Z} @>>>0,
	\end{CD}
\end{equation*}
The leftmost vertical arrow is injective whereas the rightmost vertical arrow is the trivial map sending everything to zero.  Thus, the snake lemma gives the exact sequence
$$0 \longrightarrow \mathbb{Z} \cdot N_{G} \stackrel{s}{\longrightarrow} \mathbb{Z} \longrightarrow I_{G}/T(I_{G}) \longrightarrow \mathbb{Z}[G]/T(\mathbb{Z}[G]) \stackrel{s}{\longrightarrow} \mathbb{Z} \longrightarrow 0, $$
from which we obtain the short exact sequence
$$0 \longrightarrow \mathbb{Z}/d\mathbb{Z} \longrightarrow I_{G}/T(I_{G}) \longrightarrow I_{G}/T(\mathbb{Z}[G]) \longrightarrow 0.$$
It follows that
\begin{equation} \label{unn}
|I_{G}/T(\mathbb{Z}[G])| = \frac{|I_{G}/T(I_{G})|}{d}. 
\end{equation}
Since $T:I_{G} \longrightarrow I_{G}$ is injective, we have $|I_{G}/T(I_{G})| = |{\rm det}_{\mathbb{Z}}(T)|$, but
\begin{equation*}
{\rm det}_{\mathbb{Z}}(T) = {\rm det}_{\mathbb{C}}(T_{\mathbb{C}}) = \prod_{\chi \neq \chi_{1}}L_{Y/X}^{*}(1,\chi) = \frac{\zeta_{Y}^{*}(1)}{\zeta_{X}^{*}(1)},
\end{equation*}
where this last equality is true by Theorem \ref{product_form}.  Using Theorem \ref{analogue_dir}, we obtain
\begin{equation} \label{deuxx}
|I_{G}/T(I_{G})| =  2^{(d-1)(r_{X} - 1)} \cdot d \cdot \frac{\kappa_{Y}}{\kappa_{X}},
\end{equation}
if $X \neq C_{n}$.  If $X = C_{n}$, then $Y = C_{dn}$ necessarily and using Remark \ref{cycle}, one gets instead
$$|I_{G}/T(I_{G})| =  \frac{2(dn)^{2}}{2n^{2}} = d^{2} = d \cdot \frac{\kappa_{Y}}{\kappa_{X}},$$
since $\kappa_{C_{m}}=m$.  So $(\ref{deuxx})$ is also true when $r_{X} = 1$.  Putting (\ref{unn}) and (\ref{deuxx}) together gives the desired result.
\end{proof}

This last theorem implies in particular that $2^{(d-1)(r_{X} - 1)} \cdot \frac{\kappa_{Y}}{\kappa_{X}} \in \mathbb{Z}$, but we shall now show that $\kappa_{X} \, | \, \kappa_{Y}$.  Let us introduce two maps ${\rm res}: {\rm Div}(Y) \longrightarrow {\rm Div}(X)$ defined via $w \mapsto {\rm res}(w) = v$, where $\pi(w) = v$ and ${\rm cor}: {\rm Div}(X) \longrightarrow {\rm Div}(Y)$ defined via
$$v \mapsto {\rm cor}(v) = \sum_{\substack{w \in V_{Y} \\ \pi(w)=v}} w. $$
Both ${\rm res}$ and ${\rm cor}$ are $\mathbb{Z}[G]$-module morphisms satisfying
$${\rm res} \circ {\rm cor} = |G| \cdot {\rm id}_{{\rm Div}(X)} \text{ and } {\rm cor} \circ {\rm res} = N_{G} \cdot {\rm id}_{{\rm Div}(Y)}.$$
Note that ${\rm res}$ is surjective and ${\rm cor}$ is injective.
\begin{theorem} \label{commutativity}
The following two commutative diagrams commute:
\begin{equation*} 
	\begin{CD}
		 {\rm Div}(Y)    @>{\phi_{Y}}>>   {\rm Div}(Y) \\
		 @VV{{\rm res}}V       @VV{{\rm res}}V    &  \\
		 {\rm Div}(X)    @>{\phi_{X}}>>   {\rm Div}(X)
	\end{CD}
\, \, \text{ and } \, \, 
        \begin{CD}
		 {\rm Div}(X)    @>{\phi_{X}}>>   {\rm Div}(X) \\
		 @VV{{\rm cor}}V       @VV{{\rm cor}}V    &  \\
		 {\rm Div}(Y)    @>{\phi_{Y}}>>   {\rm Div}(Y)
	\end{CD}
\end{equation*}
\end{theorem}
\begin{proof}
The commutativity of both diagrams follow from the following equality
$$\rho_{v}(v_{0}) = \sum_{w \in \pi^{-1}(v)} \rho_{w}(w_{0}), $$
where $w_{0}$ is any place in the fiber of $v_{0}$.  The details are left to the reader.
\end{proof}

The restriction map ${\rm res}:{\rm Div}(Y) \longrightarrow {\rm Div}(X)$ induces the commutative diagram
\begin{equation*} 
	\begin{CD}
		 0    @>>>   {\rm Div}^{\circ}(Y)     @>>>          {\rm Div}(Y)       @>{{\rm deg}}>> \mathbb{Z}  @>>>0 \\
		 & &  @V{${\rm res}$}VV   @V{${\rm res}$}VV         @V{\rotatebox{90}{$\simeq$}}V{${\rm id}$}V    &  \\
		 0    @>>>    {\rm Div}^{\circ}(X)      @>>>          {\rm Div}(X)      @>{{\rm deg}}>> \mathbb{Z} @>>>0,
	\end{CD}
\end{equation*}
and the snake lemma implies that the $\mathbb{Z}[G]$-module morphism ${\rm res}:{\rm Div}^{\circ}(Y)\longrightarrow {\rm Div}^{\circ}(X)$ is also surjective, since the middle map is surjective.  Combining this with Theorem \ref{commutativity}, we obtain that the morphism of $\mathbb{Z}[G]$-modules ${\rm res}$ induces a surjective morphism of $\mathbb{Z}[G]$-modules ${\rm res}:{\rm Jac}(Y) \longrightarrow {\rm Jac}(X)$.  We then have:
\begin{corollary} \label{divisibility}
If $Y/X$ is an abelian cover of multigraphs, then $\kappa_{X} \, | \, \kappa_{Y}$.
\end{corollary}

Similarly, the map ${\rm cor}$ induces a morphism of $\mathbb{Z}[G]$-modules ${\rm cor}:{\rm Jac(X)} \longrightarrow {\rm Jac}(Y)$ which we will now show to be injective.  Indeed, the corestriction map ${\rm cor}:{\rm Div}(X) \longrightarrow {\rm Div}(Y)$ satisfies ${\rm cor}({\rm Div}(X)) \subseteq {\rm Div}(Y)^{G}$.  Therefore, it induces the commutative diagram
\begin{equation} \label{on_way_norm}
	\begin{CD}
		 0    @>>>   {\rm Pr}(X)     @>>>          {\rm Div}^{\circ}(X)       @>>> {\rm Jac}(X)  @>>>0 \\
		 & &  @V{${\rm cor}$}VV   @V{${\rm cor}$}V{\rotatebox{90}{$\simeq$}}V         @V{${\rm cor}$}VV    &  \\
		 0    @>>>    {\rm Pr}(Y)^{G}      @>>>          {\rm Div}^{\circ}(Y)^{G}      @>>> {\rm Jac}(Y)^{G}
	\end{CD}
\end{equation}
whose rows are exact and whose middle vertical map is an isomorphism.  Embedding $\mathbb{Z}$ diagonally into ${\rm Div}(Y)$ leads to the short exact sequence of $\mathbb{Z}[G]$-modules
$$0 \longrightarrow \mathbb{Z} \longrightarrow {\rm Div}(Y) \stackrel{\phi_{Y}}{\longrightarrow} Pr(Y) \longrightarrow 0, $$
where $G$ acts trivially on $\mathbb{Z}$.  We then get a long exact sequence in cohomology that starts as follows:
$$0 \longrightarrow \mathbb{Z} \longrightarrow {\rm Div}(Y)^{G} \stackrel{\phi_{Y}}{\longrightarrow} {\rm Pr}(Y)^{G} \longrightarrow H^{1}(G,\mathbb{Z})\longrightarrow \ldots$$
But, since $G$ acts trivially on $\mathbb{Z}$, one has $H^{1}(G,\mathbb{Z}) = {\rm Hom}_{\mathbb{Z}}(G,\mathbb{Z}) = 0$.  It follows that $\phi_{Y}({\rm Div}(Y)^{G}) = {\rm Pr}(Y)^{G}$.  Since ${\rm Div}(Y)^{G} = {\rm cor}({\rm Div}(X))$, Theorem \ref{commutativity} implies that ${\rm Pr}(Y)^{G} = {\rm cor}({\rm Pr}(X))$.  Applying the snake lemma to (\ref{on_way_norm}) shows that the morphism of $\mathbb{Z}[G]$-modules ${\rm cor}:{\rm Jac}(X) \longrightarrow {\rm Jac}(Y)$ is injective as we wanted to show.  As a result, we can re-express Theorem \ref{index} as follows.
\begin{theorem} 
Let $Y/X$ be an abelian cover of multigraphs with automorphism group $G$.  Let us define the subgroup ${\rm Jac}^{0}(Y) = \{[D] \, | \, N_{G} \cdot [D] = 0 \}$.  Then
$$|{\rm Jac}^{0}(Y)| = \frac{\kappa_{Y}}{\kappa_{X}} \text{ and thus } |I_{G}/\theta_{Y/X}^{*}(1)\cdot e \cdot \mathbb{Z}[G]| = 2^{(d-1)(r_{X} - 1)} \cdot |{\rm Jac}^{0}(Y)|.$$
\end{theorem}
\begin{proof}
We have $|{\rm Jac}^{0}(Y)| = |{\rm ker}(N_{G})| = |{\rm ker}({\rm cor} \circ {\rm res})| = |{\rm ker}(res)| = \frac{\kappa_{Y}}{\kappa_{X}}$, where all these maps are viewed as being defined and taking values into Jacobians.
\end{proof}

\begin{remark}
Note that it follows from the proof of the last theorem that the restriction map induces an isomorphism ${\rm Jac}(Y)/{\rm Jac}^{0}(Y) \stackrel{\simeq}{\longrightarrow} {\rm Jac}(X)$.
\end{remark}

\bibliographystyle{plain}
\bibliography{main}

\end{document}